\documentclass{article}
\usepackage[utf8]{inputenc}
\usepackage{amsmath,amsthm,amssymb}
\usepackage{mathrsfs,comment}
\usepackage{mathtools}
\usepackage{xspace}
\newcommand{\R}{\mathbb{R}}
\newcommand{\C}{\mathbb{C}}
\newcommand{\Z}{\mathbb{Z}}
\newcommand{\h}{\mathbb{H}}

\newcommand{\Spin}{\mathrm{Spin}}
\newcommand{\abs}[1]{\lvert #1 \rvert}
\DeclareMathOperator{\Span}{Span}
\DeclareMathOperator{\Cl}{Cl}

\DeclareMathOperator{\CCl}{\mathbb{C}l}
\DeclareMathOperator{\Ric}{Ric}

\DeclareMathOperator{\tr}{tr}
\DeclareMathOperator{\id}{id}
\DeclareMathOperator{\scal}{scal}

\title{Killing spinors and hypersurfaces}
\author{Diego Conti and Romeo Segnan Dalmasso}

\newtheorem{theorem}{Theorem}[section]
\newtheorem{lemma}[theorem]{Lemma}

\newtheorem{corollary}[theorem]{Corollary}
\newtheorem{proposition}[theorem]{Proposition}
\theoremstyle{definition}
\newtheorem{example}[theorem]{Example}
\theoremstyle{remark}
\newtheorem{remark}{Remark}[section]

\begin{document}

\title{Killing spinors and hypersurfaces}
\author{Diego Conti and Romeo Segnan Dalmasso}

\maketitle
\begin{abstract}
We consider spin manifolds with an Einstein metric, either Riemannian or indefinite, for which there exists a Killing spinor. We describe the intrinsic geometry of nondegenerate hypersurfaces in terms of a PDE satisfied by a pair of induced spinors, akin to the generalized Killing spinor equation.

Conversely, we prove an embedding result for real analytic pseudo-Riemannian manifolds carrying a pair of spinors satisfying this condition.
\end{abstract}

\renewcommand{\thefootnote}{\fnsymbol{footnote}}
\footnotetext{\emph{MSC class 2020}: \emph{Primary} 53C25; \emph{Secondary} 53B30, 53C27, 53C50, 58J60}
\renewcommand{\thefootnote}{\arabic{footnote}}

\section{Introduction} 
A spinor $\Psi$ on a pseudo-Riemannian spin manifold $(Z,h)$ is Killing if it satisfies
\begin{equation}
\label{eqn:killing}
\nabla_X\Psi = \lambda X\cdot\Psi,
\end{equation}
where $\lambda$ is a complex constant and $\cdot$ denotes Clifford multiplication; this includes the case $\lambda=0$, when the the spinor is parallel. Unless $\Psi$ is identically zero, the condition \eqref{eqn:killing} forces the scalar curvature to be $4n(n-1)\lambda^2$. This puts strong constraints on the geometry, and forces $\lambda$ to be either real or purely imaginary.

Riemannian and pseudo-Riemannian manifolds with a Killing spinor have been studied in general relativity  since \cite{Walker1970OnSpacetimes}, and later on in supersymmetry (see \cite{Duff1986Kaluza-KleinSupergravity}). Mathematically, Killing spinors on a compact Riemannian manifold with positive curvature correspond to eigenvectors for the Dirac operator which attain the lowest possible eigenvalue (see \cite{Friedrich1980DerSkalarkrummung}); in both the Riemannian and indefinite case, Killing and parallel spinors are studied in connection with holonomy (see \cite{Wang1989ParallelForms,Bar1993RealHolonomy,Baum2014OnManifolds}). 

The geometry of Riemannian manifolds with a Killing spinor is quite rigid (see \cite{Bar1993RealHolonomy,Baum1989CompleteSpinors}). In particular, the metric is Einstein, and Ricci-flat when $\lambda=0$. This is not always true for indefinite metrics. For instance, \cite{Bohle2003KillingSO} shows that there exist non-Einstein Lorentzian manifolds admitting a Killing spinor with $\lambda$ imaginary; on the other hand, it is shown in the same paper that Lorentzian manifolds admitting Killing spinors with $\lambda$ real are necessarily Einstein. In this paper, we will focus on the case in which the metric is Einstein. More precisely, we study nondegenerate hypersufaces inside an Einstein pseudo-Riemannian manifold with a Killing spinor.

\smallskip
The case in which the spinor $\Psi$ is parallel has been considered in \cite{Bar2005}, showing that the restriction of $\Psi$ to the hypersurface satisfies
\begin{equation}
\label{eqn:generalizedkilling}
\nabla_X\psi = \frac12 A(X)\cdot\psi,
\end{equation}
where $A$ is the Weingarten operator, under the assumption that the normal is space-like; one then says that $\psi$ is a generalized Killing spinor. If the normal is time-like, one obtains the similar equation
\begin{equation}
\label{eqn:imaginarygeneralizedkilling}
\nabla_X\psi = \frac i2 A(X)\cdot\psi;
\end{equation}
then $\psi$ is called a generalized imaginary Killing spinor, or imaginary W-Killing spinor (see e.g. \cite{Baumleistnerlischewski}).

A natural problem to consider is whether this is a characterization, i.e. whether 
any pseudo-Riemannian manifold $(M,g)$  with a generalized (imaginary) Killing spinor $\psi$ can be embedded isometrically as a hypersurface inside a manifold $(Z,h)$ with a parallel spinor extending $\psi$, with the symmetric tensor $A$ corresponding to the Weingarten operator. We will then refer to $(Z,h)$ as an extension of $(M,g)$.

Little is known for the case of general signature; the main result was proved in \cite{Bar2005}, where the extension is shown to exist in the case that  $\nabla A$ is totally symmetric.

For Riemannian extensions of Riemannian metrics, the classification of the holonomy groups of a manifold with a parallel spinor in \cite{Wang1989ParallelForms} allows one to recast the problem in terms of $G$-structures and differential forms. Extending the metric  amounts to solving appropriate evolution equations in the sense of \cite{Hitchin2001StableMetrics} (see also \cite{ContiSalamon, Conti2010Calabi-YauReduction}); for some instances of $G$, the existence of a solution can then be proved using the integrability of an exterior differential system associated to the $G$-structure (see \cite{ContiSalamon,Bryant2010NonEmbeddingHolonomy}). A general proof of the existence of the embedding for real analytic data, using spinors rather than differential forms,  was given in  \cite{Ammann2013TheCP}. We point out that there are examples of non-real-analytic Riemannian manifolds with a generalized Killing spinor which cannot be extended (see \cite{Bryant2010NonEmbeddingHolonomy} and \cite[Theorem~4.27]{Ammann2013TheCP}), indicating that the real analytic condition cannot be eschewed in this context.

For Lorentzian extensions of Riemannian metrics, a proof of  existence was given in \cite{Baumleistnerlischewski} for real analytic data, and \cite{Lischewski2015TheSystem} for smooth data, under the condition
\[U_\psi \cdot \psi =i u_\psi \psi,\]
with $U_\psi$ denoting the Riemannian Dirac current and $u_\psi$ its norm. This algebraic condition on the spinor corresponds to imposing that the parallel spinor on $Z$ is null. For the $4$-dimensional case, an alternative proof using the polyform associated to the square of the spinor was given in \cite{Murcia2022ParallelFour-manifolds}. In these results, the metric on $Z$ is not automatically Ricci-flat. In order to obtain a Ricci-flat metric, one needs to impose additional constraints involving the tensor $A$ and the scalar curvature $s$, namely
\begin{equation}
\label{eqn:ricciflatconstraints}
s=\tr A^2-(\tr A)^2, \quad d\tr A + \delta A=0.
\end{equation}
It was shown in \cite{Murcia2021ParallelHypersurfaces} that a Riemannian metric on a $3$-manifold with a generalized Killing spinor such that \eqref{eqn:ricciflatconstraints} holds can be extended to a Ricci-flat Lorentzian manifold with a parallel spinor.

\smallskip
To our knowledge, hypersurfaces inside a manifold $Z$ with a Killing spinor, $\lambda\neq0$ have only been studied in the Riemannian setting and with the language of differential forms. By \cite{Bar1993RealHolonomy}, $Z$ is either  Einstein-Sasaki, $3$-Einstein Sasaki, nearly-parallel $\mathrm{ G}_2$, nearly K\"ahler of dimension $6$ or a round sphere. The case where $Z$ has dimension three has been studied in \cite{Morel2005SurfacesSpinors}. Hypersurfaces inside a nearly-K\"ahler $6$-manifold have been studied in  \cite{Fernandez2008NearlySingularities}, where the corresponding evolution equations are also introduced. Also in this context, solving the evolution equations can be used effectively to produce explicit metrics; this approach has been instrumental in the construction of inhomogeneous nearly-K\"ahler manifolds in \cite{Foscolo2017New3-spheres}. For real analytic data, the existence of an extension for the geometries corresponding to nearly-K\"ahler, nearly-parallel $\mathrm G_2$ and Einstein-Sasaki structures on $Z$ has been proved in \cite{Conti2011EmbeddingTorsion}.

In this paper, we describe in spinorial terms the geometry of a hypersurface inside an Einstein manifold with a Killing spinor $\Psi$, generalizing \eqref{eqn:generalizedkilling} and \eqref{eqn:imaginarygeneralizedkilling} (Theorem~\ref{thm:restrictionSpace}, Theorem~\ref{thm:restrictionTime}). When the hypersurface has even dimension, the geometry of the hypersurface is described by a single equation involving the restricted spinor $\psi$, which originally appears in \cite{Morel2005SurfacesSpinors} in the context of surfaces.  In odd dimensions, a complete description requires two spinors. In the case that the orthogonal distribution to the hypersurface is space-like, generated by a unit normal $\nu$, the spinors $\psi$, $\phi$ correspond to the restrictions of $\Psi$ and $\nu\cdot\Psi$, and they satisfy
\begin{equation*}
    \begin{cases}
    \nabla_X^{ M}\psi=\frac{1}{2}A(X)\odot\psi+\lambda X\odot\phi,\\
    \nabla_X^{ M}\phi=\lambda X\odot\psi-\frac{1}{2}A(X)\odot\phi;
    \end{cases}
\end{equation*}
similar equations hold in the case that the orthogonal distribution  is time-like. The case of even dimension can be subsumed in this system by introducing a second spinor $\phi=i^{\frac{3p+q+2}2}\omega_M\odot\psi$, where $(p,q)$ is the signature and $\omega_M$ the volume form. In analogy to \eqref{eqn:ricciflatconstraints}, the fact that the ambient manifold is assumed to be Einstein implies $d\tr A+\delta A=0$. We dub the resulting geometry a \emph{harmful structure}, meaning to suggest the fact that such a structure potentially leads to a Killing spinor on the extension.

The main result of this paper is that any real analytic harmful structure can be extended to an Einstein manifold with a Killing spinor. The proof is akin to \cite{Ammann2013TheCP}; however, some work is needed to handle the more general signature. Indeed, the characterization of  real analytic hypersurfaces inside Riemannian Einstein manifolds given in \cite{Koiso1981HypersurfacesOE} extends to the pseudo-Riemannian case (Corollary~\ref{cor:embeddingSpace}, Corollary~\ref{cor:embeddingTime}). We prove that the spinors defining the harmful structure can be extended to Killing spinors on $Z$ by parallel translation relative to a suitable connection. This shows that any pseudo-Riemannian spin manifold with a harmful structure can be extended to an Einstein manifold with a Killing spinor (Theorem~\ref{thm:mainpseudoriemannian}). Our result generalizes the known results quoted above in two respects: we consider pseudo-Riemannian metrics, and we allow $\lambda\neq0$ in the Killing equation. By contrast, we restrict the geometry by imposing that the metric on $Z$ is Einstein.

Working with arbitrary signature forces us to restrict to the real analytic case. However, we note that the results of \cite{DeTurck1983TheCP} indicate that smooth Riemannian hypersurfaces inside a Lorentzian Einstein manifold have a characterization similar to  \cite{Koiso1981HypersurfacesOE}. This indicates that the harmful structure corresponding to this signature always extends to an Einstein Lorentzian manifold. Determining whether the spinor can be extended to obtain a Killing spinor will be the object of future work.

\section{Hypersurfaces in Einstein manifolds}
In this section we recall Koiso's characterization  of $n$-dimensional real analytic pseudo-Riemannian manifolds $(M,g)$ which can be immersed as hypersurfaces in an Einstein manifold (see \cite{Koiso1981HypersurfacesOE}). Whilst Koiso works in Riemannian signature, the proof works in the same way for arbitrary signature, though statements need to be adapted slightly.

The Einstein manifold will take the form of a generalized cylinder in the sense of \cite{Bar2005}, i.e. a product $Z=M\times(a,b)$ endowed with a metric of the form $g_t+dt^2$, with $\{g_t\}$ a one-parameter family of metrics on $M$. In calculations, we will often  drop the subscript $t$ for simplicity. To be precise, this description will be valid locally, i.e. around each point $x$ there will be an open neighbourhood $U$ such that $U\times(a,b)$ is contained in $Z$, and the metric can be written down as a generalized cylinder; however, the interval of definition may shrink to zero as the point $x$ moves, if $M$ is not compact.

The isometric embedding of $(M,g)$ in the generalized cylinder will be obtained by imposing the initial condition $g_0=g$. The Einstein condition is a PDE which can be expressed purely in terms of $\{g_t\}$; however, it will be convenient to write it in terms of both $\{g_t\}$ and the Weingarten operators $\{A_t\}$, with the convention that for $X,Y$ tangent to $M$ the normal component of $\nabla^Z_X Y$ is  $g_t(A_t(X),Y)\frac{\partial}{\partial t}$, so $A_t=-\nabla \frac{\partial}{\partial t}$.  Notice that setting $\dot g_t(X,Y)=\frac{d}{dt}(g_t(X,Y))$, one has
\begin{equation}
\label{eqn:WeingartenInGeneralizedCylinder}
\dot g_t(X,Y)=-2g_t(A_t(X),Y).
\end{equation}
We will also consider hypersurfaces with time-like normal direction, which locally take the form of generalized cylinders $g_t^2-dt^2$; in this case, $A_t=\nabla \frac{\partial}{\partial t}$ and \eqref{eqn:WeingartenInGeneralizedCylinder} has the opposite sign.

We will need to consider the operator $\delta$ acting on tensors of type $(k,h)$ as
\[\begin{cases}
(\delta T)(v_1,\dotsc, v_{k},\alpha^1,\dotsc, \alpha^{h-1}) =
-\sum_{i=1}^n(\nabla_{e_i} T)(v_1,\dotsc, v_{k},e^i,\alpha^1,\dotsc, \alpha^{h-1}),& h>0\\
(\delta T)(v_1,\dotsc, v_{k-1}) =
-\sum_{i=1}^n (\nabla_{e_i} T)((e^i)^\sharp,v_1,\dotsc, v_{k-1}),& h=0
\end{cases}
\]
Here, $\{e_i\}$ denotes any frame, and $\{e^i\}$ its dual frame. In general, we will often consider orthonormal frames $\{e_i\}$, so that the metric takes the form
\[\epsilon_1e^1\otimes e^1+\dots + \epsilon_n e^n\otimes e^n,\]
where $\epsilon_i=\pm 1$.

Notice that for vector fields, $\delta X=-\operatorname{Div} X$, and for $1$-forms $\delta \alpha=d^*\alpha$; in particular, $\Delta f=\delta(d f)$ for any function $f$.
\begin{theorem}[Koiso \cite{Koiso1981HypersurfacesOE}]\label{thm:generalKoiso}
Let $\{g_t\}$ and $\{A_t\}$ be real analytic one-parameter families of metrics (resp. symmetric (1,1) tensors) on $M$ defined on the interval $(a,b)$, satisfying
\[\begin{cases}
\dot g_t(X,Y)=-2g_t(A_t(X),Y)\\
\dot A=-\Ric(g_t) +(\tr A) A+ K\id
\end{cases}\]
Assume further that 
\begin{equation}
\label{eqn:initialcondition}
s=(n-1)K-\tr A_0^2+(\tr A_0)^2,\quad  d\tr A_0+
\delta A_0=0.
\end{equation}
Then $g_t+dt^2$ is an Einstein metric on $M\times(a,b)$ with Einstein constant $K$.
\end{theorem}

Applying the Cauchy-Kovaleskaya theorem (see \cite{Ammann2013TheCP}), one obtains:
\begin{corollary}\label{cor:embeddingSpace}
A real analytic pseudo-Riemannian manifold $(M,g)$ of signature $(p,q)$ embeds isometrically as a hypersurface in an Einstein manifold of signature $(p+1,q)$ with $\Ric^Z=K\id$ if and only if it admits a symmetric $(1,1)$ tensor $A$ such that
\begin{equation}\label{eqn:scalarConditionEmbeddingSpace}
    s=(n-1)K-\tr A^2+(\tr A)^2,\quad d\tr A+\delta A=0.
\end{equation}
\end{corollary}

\begin{corollary}\label{cor:embeddingTime}
A real analytic pseudo-Riemannian manifold $(M,g)$ of signature $(p,q)$ embeds isometrically as a hypersurface in an Einstein manifold of signature $(p,q+1)$ with $\Ric^Z=K\id$ if and only if it admits a symmetric $(1,1)$ tensor $A$ such that
\begin{equation}\label{eqn:scalarConditionEmbeddingTime}
    s=(n-1)K+\tr A^2-(\tr A)^2,\quad d\tr A+\delta A=0.
\end{equation}
\end{corollary}
\begin{proof}
Let $\tilde g=-g$ be the opposite metric, with signature $(q,p)$. Then $\tilde\Ric=-\Ric$, $\tilde s=-s$, and $\tilde\delta A=\delta A$. Write also $\tilde A=-A$. Then 
\[\tilde s=(n-1)(-K)-\tr \tilde A^2+(\tr \tilde A)^2,\quad d\tr \tilde A+\delta \tilde A=0.\]
Therefore, we obtain a metric with $\Ric^Z=-K\id$, which locally takes the form of a generalized cylinder $\tilde g_t+dt^2$. If we reverse the sign of the metric, we find a metric locally of the form $g_t-dt^2$ satisfying $\Ric^Z=K\id$.
\end{proof}

\section{Killing spinors and hypersurfaces}
In this section we study the geometry of a  hypersurface embedded in a pseudo-Riemannian manifold with a (nonzero) Killing spinor. We show that the hypersurface inherits two spinors which satisfy a coupled differential system involving a symmetric tensor $A$, which corresponds to the second fundamental form.

Let $\Cl_{p,q}$ be the Clifford algebra of signature $(p,q)$, and let $\Sigma_{p,q}$ be the complex spinor representation. By definition, $\Sigma_{p,q}$ is a representation of $\Cl_{p,q}$; if $p+q$ is even and positive, $\Sigma_{p,q}$ splits into the sum of two representations of $\Spin_{p,q}$, denoted by $\Sigma_{p,q}^+$ and $\Sigma_{p,q}^-$, which can be identified as the $\pm1$-eigenspaces  of Clifford multiplication by the volume form when $p-q$ is a multiple of $4$, or the $\pm i$-eigenspaces if $p-q$ is not a multiple of $4$.

Let $N$ be a  spin manifold of dimension $n$ endowed with a pseudo-Riemannian metric of signature $(p,q)$, and let $\Sigma N$ denote the bundle of complex spinors; recall that $\Sigma N$ splits as $\Sigma_+ N\oplus\Sigma_- N$ when $n$ is even. Clifford multiplication gives a bundle map
\[TN\otimes \Sigma N\to \Sigma N, \quad v\otimes\psi\mapsto v\cdot\psi.\]
Let $e_1,\dots,e_n$ be a positively-oriented orthonormal basis of $TN$. Recall from {\cite[Proposition 3.3]{LawsonJr.1989SpinGeometry}} that the  volume element $\omega=e_1\dotsm e_{p+q}$ in $\Cl_{p,q}$, $p+q=n$ satisfies
\[\omega^2=(-1)^{\frac{n(n+1)}2+q}, \quad e_i\omega=(-1)^{n-1}\omega e_i.\]
In other words,
\[\omega^2=\begin{cases} 1 & p-q=0,3 \mod 4\\ -1 & p-q=1,2 \mod 4\end{cases}.\]
Conventionally, if $n$ is odd, $\omega$ acts on $\Sigma_{p,q}$ as multiplication by $i^{q+n(n+1)/2}$; if $n$ is even, it acts on each of $\Sigma_{p,q}^\pm$ as multiplication by $\pm i^{q+n(n+1)/2}$.

Now suppose $(Z,h)$ is a pseudo-Riemannian spin manifold with a Killing spinor $\Psi$, i.e. $\nabla_X^Z\Psi=\lambda X\cdot \Psi$ for any vector field $X$ of $Z$, where $\lambda$ is a complex constant. As there will not be any ambiguity, we will not use distinct symbols for the covariant derivative relative to the Levi-civita connection and the one relative to the spin connection. Since we are interested in hypersurfaces of $Z$, we will denote by $n+1$ the dimension of $Z$. As the volume element $\omega_Z$ is parallel, we have
\[\nabla_X^Z (\omega_Z\cdot\Psi)=\omega_Z\cdot \nabla_X^Z\Psi = \lambda \omega_Z\cdot X\cdot\Psi=(-1)^n\lambda X\cdot (\omega_Z\cdot \Psi).\]
Thus, $\omega_Z\cdot\Psi$ is also Killing. Assume $n$ is even. Then if $\lambda\neq0$, $\Psi$ and $\omega_Z\cdot\Psi$ are necessarily independent, since they have opposite Killing numbers. In general, we can decompose $\Psi$ as $\Psi_++\Psi_-$ and hence we obtain
\[\nabla_X^Z \Psi_+=\lambda X\cdot\Psi_-,\quad \nabla_X^Z \Psi_-=\lambda X\cdot\Psi_+.\]

Let $(M,g)$ be a nondegenerate oriented hypersurface, call $\iota\colon M\to Z$ the embedding, and let $\nu$ be a normal vector field, normalized so that $h(\nu,\nu)=1$ (or $h(\nu,\nu)=-1$). We have a bundle morphism from  the complex Clifford bundle $\CCl M$ to $\iota^* \CCl Z$,
\begin{equation}
\label{eqn:cliffordtoclifford}
v\mapsto \nu \cdot v  \quad(\text{resp. }v\mapsto i\nu\cdot v), \quad v\in TM.
\end{equation}
Recall that the Clifford algebra is graded over $\Z_2$ (see e.g. \cite{LawsonJr.1989SpinGeometry}); accordingly we have a splitting $\CCl Z=\CCl^0 Z\oplus\CCl^1 Z$. The bundle map  \eqref{eqn:cliffordtoclifford} is an isomorphism onto  $\iota^*\CCl^0 Z$; indeed, it restricts to an algebra isomorphism on each fibre, which denoting by $\odot$ the multiplication by a vector in $\CCl M$ and by $\cdot$ multiplication in $\CCl Z$ satisfies
\[v\odot w\mapsto \nu\cdot v\cdot w\quad(\text{resp. }v\odot w\mapsto i\nu\cdot v\cdot w).\]
Recalling that $n$ denotes the dimension of $M$, we obtain the identifications
\begin{equation}
    \label{eqn:identifications}
\begin{aligned}
\Sigma M&=\Sigma_+ M\oplus\Sigma_- M=\iota^*\Sigma Z, & n &\text{ even}, \\
\Sigma M&=\iota^*\Sigma_+ Z, & n &\text{ odd}.
\end{aligned}
\end{equation}
Through the identifications \eqref{eqn:identifications}, the covariant derivatives of a spinor $\Psi$ and its restriction $\psi$ are related by
\begin{equation}
  \label{eqn:inducedconnection}
  \nabla^{Z}_X\Psi=\nabla^{M}_X\psi-\frac{1}{2}\nu\cdot A(X)\cdot\psi,\quad X\in TM
\end{equation}
(see equation (3.5) in \cite{Bar2005}). If $\Psi$ is Killing, then
\[\nabla^{M}_X\Psi=\frac{1}{2}\nu\cdot A(X)\cdot\psi+\lambda X\cdot \Psi.
\]
This leads to an intrinsic formula for the covariant derivative of the restriction $\psi$ in terms of the geometry of the hypersuface $M$. We will first introduce this formula assuming $p+q$ even and $\nu$ space-like, noting that the special case of surfaces in three-dimensional Riemannian manifolds was treated in
\cite{Morel2005SurfacesSpinors}. If $p+q$ is even, the volume element $\omega_Z=e_1\cdot\ldots \cdot e_n\cdot\nu$ acts as
\[\omega_Z\cdot\Psi=i^{q+(n+1)(n+2)/2}\Psi=i^{q+n/2+1}\Psi,\]
so that
\[
\begin{split}
X\cdot\Psi &= \nu\cdot X\cdot \nu\cdot \Psi=i^{-q-n/2-1} \nu\cdot X\cdot\omega_Z\cdot \nu\cdot\Psi\\
&=i^{-q-n/2+1} \nu\cdot X\cdot e_1\cdot \ldots \cdot e_n\cdot\Psi\\
&=i^{-q-n/2+1}(-1)^{n(n-1)/2+n/2} \nu\cdot X\cdot (\nu\cdot e_1)\cdot \ldots \cdot (\nu\cdot e_n)\cdot
\Psi\\
&=i^{-q-n/2+1}X\odot \omega_M\odot\Psi
\end{split}
\]
Therefore, a Killing spinor on a manifold of signature $(p+1,q)$ with $p+q$ even induces on a  hypersurface of signature $(p,q)$ a spinor satisfying
\begin{equation}
\label{eqn:morel}
\nabla^{M}_X\psi = \frac12A(X)\odot\psi + \lambda i^{\frac{2-3q-p}{2}} X\odot \omega_M \odot \psi.
\end{equation}
For $p+q$ odd, we have the following:
\begin{theorem}\label{thm:restrictionSpace}
Let $Z$ be a pseudo-Riemannian spin manifold of dimension $n+1$ and signature $(p+1,q)$, with $n$ odd,  endowed with a Killing spinor $\Psi$ such that
\[\nabla^{Z}_X\Psi=\lambda X\cdot\Psi,\quad \lambda\in\C,\]
and let $M$ be an oriented hypersurface of signature $(p,q)$, with Weingarten operator $A(X)=-\nabla_X^{Z}\nu$. Write $\Psi=\Psi_++\Psi_-$, and define spinors $\psi$ and $\phi$ on $M$ by restricting $\Psi_+$ and $\nu\cdot\Psi_-$ and applying the isomorphism \eqref{eqn:identifications}. Then  $\phi$ and $\psi$ satisfy the coupled differential system
\begin{equation} \label{eqn:systemSpace}
    \begin{cases}
    \nabla_X^{M}\psi=\frac{1}{2}A(X)\odot\psi+\lambda X\odot\phi\\
    \nabla_X^{M}\phi=\lambda X\odot\psi-\frac{1}{2}A(X)\odot\phi
    \end{cases}
\end{equation}
and the restriction of $\Psi$ to $M$ is given by $\psi-\nu\cdot\phi$.
\end{theorem}
\begin{proof}
Using \eqref{eqn:inducedconnection}, \eqref{eqn:cliffordtoclifford}, $\nu\cdot\nu=-1$ and the fact that Clifford multiplication by a vector interchanges $\Sigma_+$ and $\Sigma_-$, we obtain
\[
\nabla^{ M}_X\psi=\lambda \nu\cdot X\cdot(\nu\cdot\Psi_-)+\frac{1}{2} A(X)\odot\psi=\lambda X\odot\phi+\frac{1}{2}A(X)\odot\psi.
\]
Similarly, we get
\[
\nabla^{ M}_X\phi-\frac{1}{2}\nu\cdot A(X)\cdot\phi=\nabla^{ Z}_X(\nu\cdot\Psi_-)=\nu\cdot\lambda X\cdot\Psi_+- A(X)\cdot\Psi_-.
\]
Thus
\[
    \nabla^{ M}_X\phi=\frac{1}{2}\nu\cdot A(X)\cdot\phi+\lambda X\odot\psi+ A(X)\cdot\nu\cdot \phi=\lambda X\odot\psi-\frac{1}{2}A(X)\odot\phi.
\]
\end{proof}
In the case that the normal is time-like, \eqref{eqn:inducedconnection} still holds, but \eqref{eqn:cliffordtoclifford} gives
 \[\nabla^{Z}_X\Psi=\nabla^{ M}_X\psi-\frac{1}{2}\nu\cdot A(X)\cdot\psi=\nabla^{ M}_X\psi+\frac{i}{2} A(X)\odot\psi;
 \]
with appropriate definitions and similar computations one obtains a similar system with a factor $-i$, i.e.
\begin{theorem}
\label{thm:restrictionTime}
Let $Z$ be a pseudo-Riemannian spin manifold of dimension $n+1$ and signature $(p,q+1)$, with $n$ odd, endowed with a Killing spinor $\Psi$, so that
\[\nabla^{ Z}_X\Psi=\lambda X\cdot\Psi,\quad \lambda\in\C,\]
and let $M$ be an oriented hypersurface of signature $(p,q)$ with Weingarten operator $A(X)=\nabla_X^{ Z}\nu$.

Write $\Psi=\Psi_++\Psi_-$, and define spinors $\psi$ and $\phi$ on $M$ by restricting $\Psi_+$ and $-\nu\cdot\Psi_-$ and applying the isomorphism \eqref{eqn:identifications}. Then $\phi$ and $\psi$ satisfy the coupled differential system
\begin{equation} \label{eqn:systemTime}
    \begin{cases}
    \nabla_X^{ M}\psi=-\frac{i}{2}A(X)\odot\psi-i\lambda X\odot\phi\\
    \nabla_X^{ M}\phi=i\lambda X\odot\psi+\frac{i}{2}A(X)\odot\phi,
    \end{cases}
\end{equation}
and the restriction of $\Psi$ to $M$ is given by $\psi-\nu\cdot\phi$.
\end{theorem}
\begin{remark}
\label{rem:evenharmful}
In the  even case, \eqref{eqn:systemSpace} and \eqref{eqn:systemTime} still hold if one sets  $\phi =i^{\frac{2-3q-p}{2}} \omega_M\odot\psi$, where $\omega_M$ is the volume form in $M$. In this case, $\phi$ is the restriction of $\nu\cdot\Psi$ (respectively $-\nu\cdot\Psi$) under the  isomorphism \eqref{eqn:identifications}. Notice that this is simply a different way of writing \eqref{eqn:morel} or its time-like analogue.
\end{remark}
Recall that the constant $\lambda$ appearing in the Killing spinor equation, and hence  equations \eqref{eqn:systemSpace} and \eqref{eqn:systemTime}, is either real or purely imaginary.

These equations characterize a geometry that gives rise to a Killing spinor in one dimension higher, but only potentially; in accord with their ``potentially killing nature'', we will dub them \emph{harmful} structures. More precisely, given a pseudo-Riemannian spin manifold $(M,g)$ of signature $(p,q)$, we will say  that a
\emph{weakly harmful structure} on $(M,g)$ is a pair of nowhere vanishing spinors $(\phi,\psi)$ satisfying either \eqref{eqn:systemSpace} or \eqref{eqn:systemTime} for some symmetric tensor $A$ and some constant $\lambda$, either real or purely imaginary; if $p+q$ is even, we further require that $\phi=i^{\frac{2-3q-p}{2}}\omega_M\odot\psi$,
where $\omega_M$ is the volume form. The weakly harmful structure will be called  \emph{real} if \eqref{eqn:systemSpace} holds and  \emph{imaginary} if \eqref{eqn:systemTime} holds.  If the symmetric tensor $A$ additionally satisfies
\[d\tr A + \delta A=0,\] 
$(\phi,\psi)$ will be called a \emph{harmful} structure.

\begin{remark}
We will see in Corollary~\ref{cor:weakHarmfulRiemann} that, on a Riemannian manifold, a real weakly harmful structure is necessarily harmful.
\end{remark}

Theorem~\ref{thm:restrictionSpace} and its time-like counterpart, Theorem~\ref{thm:restrictionTime}, show that any nondegenerate hypersurface inside an Einstein pseudo-Riemannian manifold $(Z,h)$ endowed with a Killing spinor inherits a harmful structure. If $(Z,h)$ is not assumed to be Einstein, one obtains a weakly harmful structure (see Corollary~\ref{cor:embeddingSpace} and Corollary~\ref{cor:embeddingTime}).

Notice that for $\lambda=0$, $\psi$ satisfies an equation analogous to the generalized Killing spinor equation of \cite{Bar2005}, with a factor $-i$ if one takes  the normal to be time-like, rather than space-like.

\section{Embedding into a pseudo-Riemannian Einstein spin manifold}\label{sec:embedding}
In this section we show that a real analytic pseudo-Riemannian spin manifold of signature $(p,q)$ with a harmful structure can be  embedded isometrically in a pseudo-Riemannian Einstein manifold, of signature $(p+1,q)$ or $(p,q+1)$ accordingly to whether the harmful structure is real or imaginary. We will give the detailed proofs only for real harmful structures, as the imaginary case is entirely similar. 

Since in this section all spinors are on the same manifold $M$,  we will omit the symbol $\odot$ and indicate Clifford multiplication by juxtaposition.

Given a harmful structure satisfying \eqref{eqn:systemSpace} or \eqref{eqn:systemTime}, we will denote by
\[F(X,Y)=(\nabla_XA)(Y)-(\nabla_YA)(X),\]
the exterior covariant derivative of $A$ (often denoted by $d^\nabla A$). We begin with the following:
\begin{lemma}
\label{lemma:curvature}
Let $(M,g)$ be a pseudo-Riemannian spin manifold with a real (weakly) harmful structure $(\phi,\psi)$, and let $X,Y\in TM$ be two vector fields. Then the curvature of the spinor bundle of $M$ satisfies
\begin{equation}\label{eqn:spinCurvatureSpace}
\begin{split}
    \mathcal{R}^M_{XY}\psi=&\frac{1}{2}\Big(F(X,Y)+A(Y)A(X)+g\big(A(Y),A(X)\big)\Big)\psi\\
    &+2\lambda^2\big(YX+g(X,Y)\big)\psi.
\end{split}
\end{equation}
\end{lemma}
\begin{proof}
We compute 
\[
\begin{split}
    \mathcal{R}^M_{XY}\psi=&\nabla_X\nabla_Y\psi-\nabla_Y\nabla_X\psi-\nabla_{[X,Y]}\psi\\
    \overset{\eqref{eqn:systemSpace}}{=}&\nabla_X\left(\frac{1}{2}A(Y)\psi+\lambda Y\phi\right)-\nabla_Y\left(\frac{1}{2}A(X)\psi+\lambda X\phi\right)\\
    &-\frac{1}{2}A([X,Y])\psi-\lambda[X,Y]\phi\\
    =&\frac{1}{2}(\nabla_X(A(Y))\,\psi+A(Y)\nabla_X\psi)+\lambda(\nabla_XY\phi+Y\nabla_X\phi)\\
    -&\frac{1}{2}(\nabla_Y(A(X))\psi+A(X)\nabla_Y\psi)-\lambda(\nabla_YX\phi+X\nabla_Y\phi)\\
    -&\frac{1}{2}A([X,Y])\psi-\lambda[X,Y]\phi\\
    \overset{\eqref{eqn:systemSpace}}{=}&\lambda T^\nabla(X,Y)\phi+\frac{1}{2}\Big((\nabla_XA)(Y)-(\nabla_YA)(X)+A\big( T^\nabla(X,Y)\big)\Big)\psi\\
    +&\frac{1}{2}\left[A(Y)\left(\frac{1}{2}A(X)\psi+\lambda X\phi\right)+\lambda Y\left(\lambda X\psi-\frac{1}{2}A(X)\phi\right)\right]\\
    -&\frac{1}{2}\left[A(X)\left(\frac{1}{2}A(Y)\psi+\lambda Y\phi\right)+\lambda X\left(\lambda Y\psi-\frac{1}{2}A(Y)\phi\right)\right],
    \end{split}
\]
where $T^\nabla=0$ is the torsion of the Levi Civita connection. We get
\[
\begin{split}
    \mathcal{R}^M_{XY}\psi=&\frac{\lambda}{2}\big(A(Y)X-YA(X)-A(X)Y+XA(Y)\big)\phi\\
    +&\Big(\frac{1}{2}F(X,Y)+\frac{1}{4}\big(A(Y)A(X)-A(X)A(Y)\big)+\lambda^2\big(YX-XY\big)\Big)\psi
\end{split}
\]
Since for the Clifford product $vw+wv=-2g(v,w)$ and $A$ is self-adjoint, the coefficient of $\phi$ equals zero and the statement follows.
\end{proof}
The following Lemma gives an expression for the Ricci tensor on a manifold for which the curvature satisfies \eqref{eqn:spinCurvatureSpace}.
\begin{lemma}
\label{lemma:ricci}
Assume that $(M,g)$ is an $n$-dimensional pseudo-Riemannian spin manifold with a real (weakly) harmful structure and fix an orthonormal frame $(e_1,\dots,e_n)$ for $TM$. Then the Ricci operator of $M$ satisfies 
\[
\begin{split}
    \Ric(X)\psi=&\Big(4(n-1)\lambda^2X+(\tr A)A(X)-A^2(X)\Big)\psi\\
    &+\left(\nabla_X(\tr A)+\sum_{k=1}^n\epsilon_ke_k(\nabla_{e_k}A)(X)\right)\psi.
\end{split}
\]
\end{lemma}
\begin{proof}
It is known, for example from equation (1.13) of \cite{bookb}, that
\begin{equation}\label{eqn:RicciAndCurvature2}
    \Ric(X)\psi=-2\sum_{k=1}^n\epsilon_ks_k\mathcal{R}_{Xs_k}\psi,
\end{equation}
which holds for any orthonormal basis $(s_1,\dots,s_n)$, where $\langle s_i,s_j\rangle=\epsilon_i\delta_{ij}$. Fix now an orthonormal frame $(e_1,\dots,e_n)$ for $TM$, so that $\langle e_i,e_j\rangle=g_{ij}=\epsilon_i\delta_{ij}$.
\[
\begin{split}
    \Ric(X)\psi=&-2\sum_{k=1}^n\epsilon_ke_k\bigg(2\lambda^2(e_kX+g(X,e_k))+\frac{1}{2}F(X,e_k)\\
    &+\frac{1}{2}\Big(A(e_k)A(X)+g\big(A(e_k),A(X)\big)\Big)\bigg)\psi\\
    =&\bigg(4\lambda^2(n-1)X-\sum_{k=1}^n\epsilon_ke_k\Big((\nabla_XA)(e_k)-(\nabla_{e_k}A)(X)\Big)\\
    &-\sum_{k=1}^n\epsilon_ke_k\Big(A(e_k)A(X)+g\big(A(e_k),A(X)\big)\Big)\bigg)\psi.
    \end{split}
\]
Recall that for any symmetric tensor $W$ we have
\begin{equation}
\label{eqn:trsymmetrictensor}   
\sum_{i=1}^n\epsilon_ie_i W(e_i)=-\tr(W).
\end{equation}
Then we get
\[
\begin{split}
    \Ric(X)\psi=&\Bigg(4(n-1)\lambda^2X-\sum_{k=1}^n\epsilon_ke_k\Big(A(e_k)A(X)+g\big(e_k,A^2(X)\big)\Big)\\
    &+\Big(\tr(\nabla_XA)+\sum_{k=1}^n\epsilon_ke_k(\nabla_{e_k}A)(X)\Big)\Bigg)\psi\\
    =&\bigg(4(n-1)\lambda^2X+(\tr A)A(X)-A^2(X)\\
    &+\Big(\nabla_X(\tr A)+\sum_{k=1}^n\epsilon_ke_k(\nabla_{e_k}A)(X)\big)\bigg)\psi.
\end{split}
\]
\end{proof}
The next lemma relates the scalar curvature of $M$ to the tensor $A$.
\begin{lemma}\label{lemma:scalarCurvatureSpace}
Let $(M,g)$ be  pseudo-Riemannian spin manifold endowed with a real (weakly) harmful structure $(\phi,\psi)$. Then 
\begin{equation}\label{eqn:scalarcurvatureSpace}
    \scal^g\psi=(4n(n-1)\lambda^2-\tr(A^2)+(\tr A)^2)\psi-2(d\tr A+\delta^gA)\psi.
\end{equation}
\end{lemma}
\begin{proof}
By \eqref{eqn:trsymmetrictensor}, we can write
\[
\begin{split}
    -\scal^g\psi=\sum_{j=1}^n\epsilon_je_j\Ric(e_j)\psi=\sum_{j=1}^n&\epsilon_je_j\bigg[4(n-1)\lambda^2e_j+(\tr A)A(e_j)-A^2(e_j)\\
    &+\Big(\nabla_{e_j}(\tr A)+\sum_{k=1}^n\epsilon_ke_k\,(\nabla_{e_k}A)(e_j)\Big)\bigg]\psi.
\end{split}
\]
All terms are straightforward to compute, except the last one that yields 
\[
\begin{split}
    \sum_{j=1}^n\epsilon_je_j\sum_{k=1}^n\epsilon_ke_k\,&(\nabla_{e_k}A)(e_j)= \sum_{j,k=1}^n\epsilon_j\epsilon_k(-e_k\,e_j-2\langle e_j,e_k\rangle)\,(\nabla_{e_k}A)(e_j)\\
    =&-\sum_{j,k=1}^n\epsilon_k\Big(e_k\,\big(\epsilon_je_j\,(\nabla_{e_k}A)(e_j)\big)+2\epsilon_j^2\delta_{jk}\,(\nabla_{e_k}A)(e_j)\Big)\\
    =&\sum_{k=1}^n \epsilon_ke_k\tr\big(\nabla_{e_k}A\big)-2\sum_{k=1}^n\epsilon_k(\nabla_{e_k}A)(e_k)\\
    =&d\tr A+2\delta^g A.
\end{split}
\]
Putting everything together we get
\[
\scal^g\psi=\big(4n(n-1)\lambda^2+(\tr A)^2-\tr(A^2)\big)\psi-2\big(d\tr A+\delta^gA\big)\psi.
\]
\end{proof}
As an immediate consequence we have the following
\begin{corollary}\label{cor:weakHarmfulRiemann}
On a Riemannian spin manifold, any real weakly harmful structure is harmful.
\end{corollary}
\begin{proof}
We write \eqref{eqn:scalarcurvatureSpace} as
\[
\big(\scal^g-4n(n-1)\lambda^2-(\tr A)^2+\tr(A^2)\big)\psi=-2(d\tr A+\delta^gA)\psi;
\]
this equation has the form  $f\psi=X\psi$, which implies that $f=X=0$ since
\[
f^2\psi=fX\psi=Xf\Psi=X^2\psi=-\abs{X}^2\psi,
\]
and $\psi$ is nowhere zero.
Thus 
\[
\scal^g=4n(n-1)\lambda^2-\tr(A^2)+(\tr A)^2,\quad d\tr A+\delta^gA=0.\]
\end{proof}
\begin{remark}
Notice that positive definiteness of $g$ is essential in the proof of Corollary~\ref{cor:weakHarmfulRiemann}, as otherwise the vanishing of $\abs{X}^2$ would not imply the vanishing of $X$. Notice also that  considering an imaginary weakly harmful structure rather than real one would make an  imaginary unit appear, invalidating the argument.
\end{remark}
Analogous results to Lemma~\ref{lemma:curvature}, Lemma~\ref{lemma:ricci} and Lemma~\ref{lemma:scalarCurvatureSpace} can be proved for imaginary harmful structures; the proofs are completeley analogous. We summarize these results in the following:
\begin{lemma}
Assume that $(M,g)$ is an $n$-dimensional pseudo-Riemannian spin manifold with an imaginary harmful structure. Then:
\begin{itemize}
    \item the curvature  satisfies
    \[
    \begin{split}
    \mathcal{R}^M_{XY}\psi=&\frac{1}{2}\Big(A(X)A(Y)+g\big(A(Y),A(X)\big)-iF(X,Y)\Big)\psi\\
    &+2\lambda^2(YX+g(X,Y))\psi;
    \end{split}
    \]
    \item the Ricci operator  satisfies
    \[
    \begin{split}
    \Ric(X)\psi=&\Big(4(n-1)\lambda^2X-(\tr A)A(X)+A^2(X)\Big)\psi\\
    &-i\left(\nabla_X(\tr A)+\sum_{k=1}^n\epsilon_ke_k(\nabla_{e_k}A)(X)\right)\psi;
    \end{split}
    \]
    \item\label{lemma:scalarCurvatureTime} the scalar curvature  satisfies
    \[
    \scal^g\psi=(4n(n-1)\lambda^2+\tr(A^2)-(\tr A)^2)\psi+2i(d\tr A+\delta^gA)\psi.
    \]
\end{itemize}
\end{lemma}

We can now prove the main result of this section. It can be viewed as a generalization of a result of \cite{Ammann2013TheCP} for generalized Killing spinors in Riemannian manifolds; our results differs in that it allows nonzero $\lambda$, though the proof is similar.
\begin{proposition}\label{prop:extensionMetricPseudo}
Let $(M,g)$ be a real analytic pseudo-Riemannian spin manifold of dimension $n$ and signature $(p,q)$ with a real (resp. imaginary) harmful structure $(\phi,\psi)$. Then $(M,g)$ can be  embedded isometrically in a pseudo-Riemannian Einstein manifold $(Z,h)$ of signature $(p+1,q)$ (resp. $(p,q+1))$, with constant scalar curvature $4n(n+1)\lambda^2$.
\end{proposition}
\begin{proof}
It is sufficient to apply Corollaries \ref{cor:embeddingSpace} or \ref{cor:embeddingTime} appropriately; $d\tr A+\delta^gA$ is zero  by assumption, and the scalar curvature satisfies \eqref{eqn:scalarConditionEmbeddingSpace} or \eqref{eqn:scalarConditionEmbeddingTime} thanks to  Lemma~\ref{lemma:scalarCurvatureSpace} and Lemma~\ref{lemma:scalarCurvatureTime}.
\end{proof}

\section{Extension of the spinor}
\label{sec:extendspinors}
In this section we improve the results of Section~\ref{sec:embedding}, showing that the spinors defining the harmful structure actually extend to Killing spinors on $(Z,h)$.

As the arguments for the real and imaginary case are quite similar, we will give the complete proof of only the first one, hence the unit normal $\nu$ satisfies $h(\nu,\nu)=1$. Throughout this section, we fix a real analytic pseudo-riemannian manifold $(M,g)$ with a real harmful structure $(\phi,\psi)$, and we embed $M$ into an Einstein manifold $(Z,h)$. We will extend $\nu$ to a vector field in a neighbourhood of $M$ in $Z$ by taking the tangent direction to geodesics normal to $M$. Following~\cite{Ammann2013TheCP}, we exploit the fact that a spinor on $Z$ is Killing if and only if is parallel relative to the modified connection
\[
\Tilde{\nabla}_X\Phi=\nabla^Z_X\Phi-\lambda X\cdot\Phi.
\]
Now define a spinor $\Psi$ on $Z$ by parallel transport of $\psi$ ($n$ even) or $\psi-\nu\cdot\phi$ ($n$ odd) relative to $\Tilde{\nabla}$ along the geodesics tangent to $\nu$. Clearly, since $\tilde{\nabla}_\nu\Psi=0$, we get that
\begin{equation}\label{eqn:killingnu}
    \nabla^Z_\nu\Psi=\lambda\nu\cdot\Psi.
\end{equation}
Hence we extended the spinor and we proved that it satisfies the Killing equation for $\nu$ at least. The next part is not as trivial. We start by computing $\tilde{\nabla}_X\Psi_{|(M,0)}$, i.e. the restriction of $\tilde{\nabla}_X\Psi$ to $M$ seen as a hypersurface embedded in $Z$. We need to consider the even and odd case separately: the former gives
\[
\begin{split}
    \tilde{\nabla}_X\Psi_{|(M,0)}&=\nabla_X^Z\Psi_{|(M,0)}-\lambda X\cdot\Psi=\nabla_X^M\psi-\frac{1}{2}\nu\cdot A(X)\cdot\Psi-\lambda X\cdot\Psi\\
    &=\lambda X\odot\phi+\frac{1}{2}A(X)\odot\psi-\frac{1}{2} A(X)\odot\psi-\lambda X\cdot\Psi\\
    &=\lambda\nu\cdot X\cdot(\nu\cdot\psi)-\lambda X\cdot\Psi_{|(M,0)}=\lambda X\cdot\psi-\lambda X\cdot\psi=0,
\end{split}
\]
while the latter is 
\[
\begin{split}
    \tilde{\nabla}_X\Psi_{|(M,0)}=&\nabla_X^Z\Psi_{|(M,0)}-\lambda X\cdot\Psi=\nabla_X^M(\psi-\nu\cdot\phi)-\frac{1}{2}\nu\cdot A(X)\cdot\Psi-\lambda X\cdot\Psi\\
    =&\frac12A(X)\odot\psi+\lambda X\odot\phi-\nabla_X^M\nu\odot\phi-\nu\cdot\left(\lambda X\odot\psi-\frac12A(X)\odot\phi\right)\\
    &-\frac12A(X)\odot(\psi-\nu\cdot\phi)-\lambda X\cdot(\psi-\nu\cdot\phi)=0.
\end{split}
\]
Thus, the restriction of $\Psi$ to $M$ is parallel with respect to this connection both in the even and in the odd case. Following \cite{Bar2005}, we prove that $\Psi$ is Killing by showing that $\tilde{\nabla}_X\Psi$ is zero for all vector fields $X$ on $Z$ obtained by extending a vector field on $M$ by parallel transport along $\nu$, meaning that $\nabla_\nu X=0$; this condition implies 
\[ [X,\nu]=\nabla^Z_X\nu=-A_t(X).\]
Throughout this section, the vector fields denoted by $X$ or $Y$ will be assumed to be of this type. 

In order to show that $\tilde\nabla_X\Psi$ vanishes on $Z$, it will be sufficient to prove
\begin{equation}\label{eqn:KillingQondition}
    \nabla^Z_\nu\tilde{\nabla}_X\Psi=0,
\end{equation}
as $\tilde\nabla_X\Psi$ is identically zero on $M$.

In the rest of the section we shall work on $Z$, and write for simplicity $\nabla$, $\mathcal{R}$ instead of $\nabla^Z$, $\mathcal{R}^Z$. Furthermore, we will omit the Clifford multiplication symbol $\cdot$, except in expressions such as
$\mathcal{R}_{XY} U\cdot \Psi$, which represents $(\mathcal{R}_{XY} U)\cdot \Psi$ rather than $\mathcal{R}_{XY} (U\cdot \Psi)$.
\begin{lemma}\label{lemma:conditionL}
Fix a spinor $\psi$ on $M$ and consider its extension $\Psi$ to $Z$ via $\tilde\nabla$-parallel transport along $\nu$. Then
\[
\nabla_\nu\tilde{\nabla}_X\Psi=\mathcal{R}_{\nu X}\Psi+2\lambda^2\nu X\Psi+\lambda\nu\tilde{\nabla}_X\Psi+\tilde{\nabla}_{A(X)}\Psi.
\]
\end{lemma}
\begin{proof}
We have
\[
\begin{split}
    \nabla_\nu\tilde{\nabla}_X\Psi&=\nabla_\nu\nabla_X\Psi-\lambda\big(\nabla_\nu X \Psi+X\nabla_\nu\Psi\big)=\nabla_\nu\nabla_X\Psi-\lambda^2X\nu\Psi\\
    0=\nabla_X\tilde{\nabla}_\nu\Psi&=\nabla_X\nabla_\nu\Psi-\lambda\big(\nabla_X \nu \Psi+\nu\nabla_X\Psi\big)=\nabla_X\nabla_\nu\Psi+\lambda\big(A(X)\Psi-\nu\nabla_X\Psi\big).
\end{split}
\]
Thus, subtracting the second one from the first we obtain
\[
\begin{split}
\nabla_\nu\tilde{\nabla}_X\Psi&=\mathcal{R}_{\nu X}\Psi+\nabla_{A(X)}\Psi-\lambda^2X\nu\Psi- \lambda\big(A(X)\Psi-\nu\nabla_X\Psi\big)\\
&=\mathcal{R}_{\nu X}\Psi+\lambda\big(\lambda\nu X\Psi+\nu\nabla_X\Psi\big)+\tilde{\nabla}_{A(X)}\Psi\\
&=\mathcal{R}_{\nu X}\Psi+2\lambda^2\nu X\Psi+\lambda\nu\tilde{\nabla}_X\Psi+\tilde{\nabla}_{A(X)}\Psi.
\end{split}
\]
\end{proof}
Recall that $(Z,h)$ is an Einstein manifold, that is $\Ric^Z=ch$, where $c=4n\lambda^2$, as seen in Section~\ref{sec:embedding}. 
Following \cite{Ammann2013TheCP}, we define  sections $L,P$ of $(\nu^\perp)^\ast\otimes\Sigma Z$ and a section $Q$ of $\bigwedge^2(\nu^\perp)^\ast\otimes\Sigma Z$ by
\begin{align*}
     P(X)=&\mathcal{R}_{\nu X}\Psi+2\lambda^2\nu X\Psi,& L(X)=&\tilde{\nabla}_X\Psi,&\\
    Q(X,Y)=&\mathcal{R}_{XY}\Psi+2\lambda^2\big(XY+\langle X,Y\rangle\big)\Psi,
\end{align*}
and note that by Lemma~\ref{lemma:conditionL}
\begin{equation}
\label{eqn:nablanuL}
    (\nabla_\nu L)(X)=\nabla_\nu\tilde{\nabla}_X\Psi=P(X)+\lambda\nu L(X)+L(A(X)).
\end{equation}
The strategy of \cite{Ammann2013TheCP} is to show that $L,P,Q$ satisfy a linear, homogeneous PDE; zero is a solution, so by uniqueness one deduces that $L$ vanishes identically. We will simplify a bit by observing that $P$ can be obtained from $Q$ by means of a contraction, so that the PDE can be expressed in terms of $L$ and $Q$ alone.

For the remainder of the section, the sum over repeated indices will be implied.
\begin{lemma}\label{lemma:contraction}
The sections $P$ and $Q$ are related by 
\[ P(X)=\nu\epsilon_j e_jQ(e_j,X).\]
\end{lemma}
\begin{proof}
Writing \eqref{eqn:RicciAndCurvature2} as $\Ric(Y)\psi=2(\epsilon_k e_k\cdot\mathcal{R}_{e_kY}+\nu \mathcal{R}_{\nu Y})\psi$,
we have
\[\begin{split}
\nu \mathcal{R}_{\nu Y}\psi &=\frac12\Ric(Y)\psi - \epsilon_j e_j \mathcal{R}_{e_jY}\psi\\
&=2\lambda^2nY\psi - \epsilon_je_jQ(e_j,Y)+2\lambda^2 \epsilon_je_j e_j Y\psi+2\lambda^2\epsilon_je_j\langle e_j,Y\rangle \psi\\
&=2\lambda^2Y\psi-\epsilon_je_j Q(e_j,Y),
\end{split}\]
and by multiplying by $\nu$ we get $P(X)=\nu\epsilon_j e_jQ(e_j,X)$.
\end{proof}
The previous lemma shows that $P$ is obtained from $Q$ by a contraction, so the right-hand side of \eqref{eqn:nablanuL} can be expressed in terms of $L$ and $Q$. The derivative of $Q$ along $\nu$ is given by the following:
\begin{proposition}\label{prop:conditionQ}
The section $Q$ satisfies 
\begin{equation}\label{eqn:equationQ}
    \nabla_\nu Q(X,Y)=\nu\epsilon_j e_j\big((\nabla_XQ)(e_j,Y)-(\nabla_YQ)(e_j,X)\big)+L_2(X,Y).
\end{equation}
where $L_2$ depends linearly on $L$ and $Q$.
\end{proposition}
\begin{proof}
By definition of $P$ and Lemma \ref{lemma:contraction}, we have
\[
\begin{split}
    \nabla_X(\mathcal{R}_{\nu Y}\Psi)=&\nabla_X(P(Y)-2\lambda^2\nu\ Y\Psi)=\nabla_X(\nu \epsilon_j e_jQ(e_j,Y)-2\lambda^2\nu Y\Psi)\\
    =&\nu \epsilon_j e_j\big((\nabla_XQ)(e_j,Y)+Q(\nabla_Xe_j,Y)+Q(e_j,\nabla_XY)\big)\\
    &-A(X)\epsilon_je_jQ(e_j,Y)+\nu\epsilon_j(\nabla_Xe_j)Q(e_j,Y)\\
    &+2\lambda^2\big(A(X)Y\Psi-\nu\nabla_XY\Psi-\nu Y\nabla_X\Psi\big)\\
    =&\nu\epsilon_je_j(\nabla_XQ)(e_j,Y)+2\lambda^2\big(A(X)Y\Psi-\nu(\nabla_XY)\Psi-\lambda\nu YX\Psi\big)\\
    &+U(X,Y),
\end{split}
\]
where 
\begin{multline*}
    U(X,Y)=
\nu \epsilon_je_j\big(Q(\nabla_Xe_j,Y)+Q(e_j,\nabla_XY)\big)\\
    -A(X)\epsilon_je_jQ(e_j,Y)+\nu\epsilon_j(\nabla_Xe_j)Q(e_j,Y)
    -2\lambda^2\nu Y L(X)
\end{multline*}
depends linearly on $L$ and $Q$; notice that $U$ also depends on the connection form.

On the other hand,
\[
\begin{split}
    \nabla_X(\mathcal{R}_{\nu Y}\Psi)=&(\nabla_X\mathcal{R})_{\nu Y}\Psi-\mathcal{R}_{A(X)Y}\Psi+\mathcal{R}_{\nu\nabla_XY}\Psi+\mathcal{R}_{\nu Y}(L(X)+\lambda X\Psi)\\
    =&(\nabla_X\mathcal{R})_{\nu Y}\Psi-Q(A(X),Y)+2\lambda^2\big(A(X)Y+\langle A(X),Y\rangle\big)\Psi\\
    &+\nu\epsilon_j e_jQ(e_j,\nabla_XY)-2\lambda^2\nu(\nabla_XY)\Psi+\mathcal{R}_{\nu Y}L(X)\\
    &+\lambda\mathcal{R}_{\nu Y}X\cdot\Psi+\lambda X\big(\nu\epsilon_j e_jQ(e_j,Y)-2\lambda^2\nu Y\Psi)\\
    =&(\nabla_X\mathcal{R})_{\nu Y}\Psi+2\lambda^2\big(A(X)Y+\langle A(X),Y\rangle\big)\Psi-2\lambda^2\nu(\nabla_XY)\Psi\\
    &+\lambda\mathcal{R}_{\nu Y}X\cdot\Psi-2\lambda^3X\nu Y\Psi+V(X,Y),
\end{split}
\]
where
\[
V(X,Y)=-Q(A(X),Y)+\nu\epsilon_j e_jQ(e_j,\nabla_XY)+\mathcal{R}_{\nu Y}L(X)+\lambda X \nu\epsilon_j e_jQ(e_j,Y)
\]
depends linearly on $L$ and $Q$; notice that $V$ also depends on the connection form and the curvature.

Equating the terms and isolating $(\nabla_X\mathcal{R})_{\nu Y}\Psi$ we get
\[
\begin{split}
    (\nabla_X\mathcal{R})_{\nu Y}\Psi=&\nu\epsilon_j e_j(\nabla_XQ)(e_j,Y)+2\lambda^2\big(A(X)Y\Psi-\nu(\nabla_XY)\Psi-\lambda\nu YX\Psi\big)\\
    &-2\lambda^2\big(A(X)Y+\langle A(X),Y\rangle\big)\Psi+2\lambda^2\nu(\nabla_XY)\Psi-\lambda\mathcal{R}_{\nu Y}X\cdot\Psi\\
    &+2\lambda^3X\nu Y\Psi+U(X,Y)-V(X,Y)\\
    =&\nu\epsilon_j e_j(\nabla_XQ)(e_j,Y)-2\lambda^2\langle A(X),Y\rangle\Psi+4\lambda^3\nu\langle X,Y\rangle\Psi\\
    &-\lambda\mathcal{R}_{\nu Y}X\cdot\Psi+S(X,Y),
\end{split}
\]
where $S(X,Y)=U(X,Y)-V(X,Y)$ depends linearly on $L$ and $Q$.

By applying the second and first Bianchi identities we obtain
\[
\begin{split}
    \nabla_\nu Q(X&,Y)=(\nabla_\nu\mathcal{R})_{XY}\Psi+\lambda\mathcal{R}_{XY}\nu\cdot\Psi+\lambda\nu\mathcal{R}_{XY}\Psi+2\lambda^3\nu\big(XY+\langle X,Y\rangle\big)\Psi\\
    =&(\nabla_X\mathcal{R})_{\nu Y}\Psi
    -(\nabla_Y\mathcal{R})_{\nu X}\Psi
    +\lambda\mathcal{R}_{XY}\nu\cdot\Psi+\lambda\nu\mathcal{R}_{XY}\Psi\\
    &+2\lambda^3\nu\big(XY+\langle X,Y\rangle\big)\Psi\\
    =&\nu\epsilon_j e_j(\nabla_XQ)(e_j,Y)-2\lambda^2\langle A(X),Y\rangle\Psi+4\lambda^3\nu\langle X,Y\rangle\Psi-\lambda\mathcal{R}_{\nu Y}X\cdot\Psi\\
    -&\nu\epsilon_j e_j(\nabla_YQ)(e_j,X)+2\lambda^2\langle A(Y),X\rangle\Psi-4\lambda^3\nu\langle Y,X\rangle\Psi+\lambda\mathcal{R}_{\nu X}Y\cdot\Psi\\
    +&\lambda\mathcal{R}_{XY}\nu\cdot\Psi
    +\lambda\nu\mathcal{R}_{XY}\Psi+2\lambda^3\nu\big(XY+\langle X,Y\rangle\big)\Psi
    +S(X,Y)-S(Y,X)\\
    =&\nu\epsilon_j e_j\big((\nabla_XQ)(e_j,Y)-(\nabla_YQ)(e_j,X)\big)+L_2(X,Y),
\end{split}\]
where $L_2(X,Y)=\lambda \nu Q(X,Y)+S(X,Y)-S(Y,X)$.
\end{proof}
We are now able to prove the main theorem, which improves Proposition \ref{prop:extensionMetricPseudo}.
\begin{theorem}
\label{thm:mainpseudoriemannian}
Assume $(M,g)$ is a real analytic pseudo-Riemannian spin manifold of signature $(p,q)$ with a harmful structure $(\phi,\psi)$. Then:
\begin{itemize}
    \item if $(\phi,\psi)$ is real, $(M,g)$  embeds isometrically in a pseudo-Riemannian Einstein spin manifold $(Z,h)$ with signature $(p+1,q)$ and Weingarten operator $A$;
    \item if $(\phi,\psi)$ is imaginary, $(M,g)$  embeds isometrically in a pseudo-Riemannian Einstein spin manifold $(Z,h)$ with signature $(p,q+1)$ and Weingarten operator $A$.
\end{itemize} In both cases $\psi$ extends to a Killing spinor $\Psi$ on $Z$ satisfying $\nabla^Z_X\Psi=\lambda X\Psi$ for any $X\in TZ$.
\end{theorem}
\begin{proof}
The isometric embedding follows from Proposition~\ref{prop:extensionMetricPseudo}; as explained at the beginning of this section, we can extend $\psi$ to a spinor $\Psi$ in such a way that \eqref{eqn:killingnu} holds. We only need to prove that $\Psi$ satisfies the Killing equation; this is equivalent to showing that $L(X)\equiv0$ on $Z$. We have already proved that $L(X)$ is zero on $M\times\{0\}$. To see that $Q$ vanishes on $M\times\{0\}$, let $X,Y$ be vector fields on $M$, and write
\begin{multline*}
Q(X,Y)=\nabla_X\nabla_Y \Psi -\nabla_Y\nabla_X \Psi -\nabla_{[X,Y]}\Psi + 2\lambda^2(XY+\langle X,Y\rangle)\Psi\\
=\nabla_X (\lambda Y\Psi)-\nabla_Y (\lambda \nabla_X \Psi)-\lambda [X,Y]\Psi  + 2\lambda^2(XY+\langle X,Y\rangle)\Psi\\
=\lambda(\nabla_X Y +\lambda YX - \nabla_Y X-\lambda XY -[X,Y]+2\lambda XY  + 2\lambda \langle X,Y\rangle)\Psi=0.
\end{multline*}
Using \eqref{eqn:nablanuL} and Proposition~\ref{prop:conditionQ} we see that $L$ and $Q$ satisfy the linear PDE system
\[
\begin{cases}
    (\nabla_\nu L)(X)=\lambda \nu L(X)+\nu \epsilon_j e_j Q(e_j,X)+L(A(X))\\
    (\nabla_\nu Q)(X,Y)=\nu\epsilon_j e_j\big((\nabla_XQ)(e_j,Y)-(\nabla_YQ)(e_j,X)\big)+L_2(L,Q).
\end{cases}
\]
By the Cauchy-Kowalewskaya Theorem we know that the solution to the PDE system is unique and, since $L=0=Q$ is a solution, it is the only one. In particular $L=0$ on $Z$ and $\Psi$ is a Killing spinor.
\end{proof}
Theorem~\ref{thm:mainpseudoriemannian} is not quite a generalization of the results of \cite{Ammann2013TheCP} for parallel spinors, in that it entails the extra hypothesis $d\tr A+\delta A=0$. However, if we restrict to the Riemannian case,  we can use Corollary \ref{cor:weakHarmfulRiemann} to remove this extra hypothesis:
\begin{corollary}
\label{cor:mainriemannian}
Assume $(M,g)$ is a real analytic Riemannian spin manifold with a real weakly harmful structure $(\phi,\psi)$. Then $(M,g)$  embeds isometrically in a Riemannian spin manifold $(Z,h)$ with Weingarten operator $A$, and $\psi$ extends to a Killing spinor $\Psi$ on $Z$ satisfying $\nabla^Z_X\Psi=\lambda X\Psi$ for any $X\in TZ$.
\end{corollary}
\begin{example}
We now present an example in the invariant setting. By considering a Lie group with a left-invariant metric, we are able to work at the Lie algebra level by extending each object by left translation. Consider the Lie algebra 
\[
\mathfrak{g}=(-2e^{23}, 3e^{13}-3e^{34},-3e^{12}+3e^{24},2e^{23})
\] 
where the notation means that for some basis $\{e^1,\dotsc, e^4\}$ of $\mathfrak{g}^*$, the Chevalley-Eilenberg operator $d$ satisfies
\[
de^1=-2e^2\wedge e^3,\,de^2=3e^1\wedge e^3-3e^3\wedge e^4,\,de^3=-3e^1\wedge e^2+3e^2\wedge e^4,\,de^4=2e^2\wedge e^3.
\]
Consider a Lie group with Lie algebra $\mathfrak{g}$, and extend the coframe $e^1,\dotsc, e^4$ by left-invariance, so that $d$ becomes the usual exterior derivative, and consider the metric
\[
g=e^1\otimes e^1+e^2\otimes e^2+e^3\otimes e^3-e^4\otimes e^4.
\]
By Theorem 4.3 of \cite{LawsonJr.1989SpinGeometry} we know that $\Cl_{3,1}=M(2,\h)$. Since we ultimately want to extend a spinor to a $5$-dimensional manifold, we also need to consider $\h^2$ as a $\Cl(4,1)$-module; writing out quaternionic matrices in complex terms, we obtain an orthonormal basis of $\R^{4,1}$ given by
\begin{gather*}
\tilde E_1=\begin{pmatrix} i & 0 &0&0\\ 0 & - i&0&0\\0&0&-i&0\\0&0&0&i\end{pmatrix},
\tilde E_2=\begin{pmatrix} 0 & -1&0&0 \\ 1 & 0&0&0\\0&0&0&1\\0&0&-1&0\end{pmatrix}, \tilde E_3=\begin{pmatrix}0 &0& 1 &0\\ 0&0&0&1\\-1 & 0&0&0\\0&-1&0&0 \end{pmatrix},\\
\tilde E_4=\begin{pmatrix}0&0 & 1&0 \\0&0&0&1\\1&0&0&0\\0&1&0&0 \end{pmatrix},\tilde E_5 =  \begin{pmatrix} 0&-i&0&0\\-i&0&0&0\\0&0&0&i\\0&0&i&0\end{pmatrix},
\end{gather*}
where the time-like vector is $\tilde E_4$. The inclusion $\Cl_{3,1}\to\Cl_{4,1}$ of \eqref{eqn:cliffordtoclifford} is then realized by $e_i\mapsto E_i=\tilde E_5\tilde E_i$.

As we need the covariant derivative we write the connection form
\[
\left(\begin{array}{cccc}0&- e^3&e^2&0\\e^3&0&-2 e^1-2 e^4&e^3\\- e^2&2 e^1+2 e^4&0&- e^2\\0&e^3&- e^2&0\end{array}\right).
\]
Consider the spinors $\psi=(i,1,i,1)$ and $\phi=i^{\frac{2-3q-p}{2}}\omega\psi=(-i,1,-i,1)$, and the endomorphism $A=e^1\otimes(2e_1-e_4)+e^2\otimes e_2+e^3\otimes e_3+e^4\otimes e_1$. We will prove that $(\phi,\psi)$ is a harmful structure, that is, they satisfy the system \eqref{eqn:systemSpace} with $\lambda=i/2$ and $d\tr A+\delta A=0$, and show that the connected, simply-connected Lie group with Lie algebra $\mathfrak{g}$ extends to a 5-dimensional Einstein manifold endowed with a Killing spinor.

The correspondence $e_i\longleftrightarrow E_i$ gives
\[
\nabla^M\psi=\frac12\Big[2e^1\otimes E_2E_3-e^2\otimes(E_1E_3+E_3E_4)+e^3\otimes(E_1E_2+E_2E_4)+2e^4\otimes E_2E_3\Big]\psi,
\]
whilst
\[
A\psi=\Big(e^1\otimes \big(2E_1-E_4\big)+e^2\otimes E_2+e^3\otimes E_3+e^4\otimes E_1\Big)\psi.
\]
We need to verify that $\nabla^M\psi=\frac12(A\psi+i\phi)$, which, noticing that
\[
(E_1-E_2E_3)\psi=(E_2-E_3E_1)\psi=(E_3-E_1E_2)\psi=0,
\]
is equivalent to the system 
\[
E_4\psi-iE_1\phi=iE_2\phi-E_3E_4\psi=E_2E_4\psi-iE_3\phi=iE_4\phi-E_2E_3\psi=0.
\]
It is easy to see that these are satisfied, hence $\nabla_X\psi=\frac12\big(A(X)\psi+iX\phi\big)$ for any $X\in\mathfrak{g}$. Since $p+q$ is even, the second equation in \eqref{eqn:systemSpace} is automatic, i.e. both equations reduce to \eqref{eqn:morel}. 

Now consider the derivation
\[D=2e^1\otimes (e_1-e_4)+e^2\otimes e_2+e^3\otimes e_3.\]
Its symmetric part coincides with $A$; it follows that the semidirect product $\tilde{\mathfrak g}=\mathfrak g\rtimes_D\Span\{e_5\}$ satisfies the equations of Theorem~\ref{thm:generalKoiso}. Explicitly, one can write
\[\tilde{\mathfrak g}= (2e^{15}-2e^{23},e^{25}+ 3e^{13}-3e^{34},e^{35}-3e^{12}+3e^{24},-2e^{15}+2e^{23},0),
\]
and verify  that the metric
\[\tilde g=e^1\otimes e^1+e^2\otimes e^2+e^3\otimes e^3-e^4\otimes e^4+e^5\otimes e^5\]
is Einstein with $\Ric=-4\id$ and the spinor $\Psi=(i,1,i,1)$ is Killing with Killing number $i/2$. In fact, this is a Lorentz-Einstein-Sasaki metric; if one reverses the sign of the metric along the Reeb vector field $e_4$,  one obtains the known $\eta$-Einstein-Sasaki metric on the Lie algebra $D_{22}$ in the classification of \cite{DIATTA2008544} (see also \cite{article,Andrada2009A5-manifolds}).
\end{example}

\medskip
\textbf{Acknowledgments}
This paper originated as part of the PHD thesis of the second author, written under the supervision of the first author, for the joint PHD programme in Mathematics Università di Milano Bicocca - University of Surrey; the authors acknowledge a partial support from the PRIN 2022MWPMAB project ``Interactions between Geometric Structures and Function Theories'' and GNSAGA of INdAM. The first author also acknowledges the MIUR Excellence Department Project awarded to the Department of Mathematics, University of Pisa, CUP I57G22000700001.

The authors are grateful to Federico A.~Rossi and the anonymous referee for many useful comments and corrections that improved the presentation of the paper.

\bibliographystyle{acm}
\bibliography{references}

\medskip
\small\noindent D. Conti: Dipartimento di Matematica, Università di Pisa, largo Bruno Pontecorvo 6, 56127 Pisa, Italy.\\
\texttt{diego.conti@unipi.it}\\
\small\noindent R.~Segnan Dalmasso: Dipartimento di Matematica e Applicazioni, Universit\`a di Milano Bicocca, via Cozzi 55, 20125 Milano, Italy.\\
\texttt{romeo.segnandalmasso@unimib.it}

\end{document}